\documentclass[12pt]{amsart}

\usepackage{amsmath}
\usepackage{amsfonts}     
\usepackage{amssymb}
\usepackage{latexsym}
\usepackage{graphicx}
\usepackage{enumerate}
\usepackage{multirow}
\usepackage{subfigure}

\newtheorem{theorem}{Theorem}

\newtheorem{lemma}[theorem]{Lemma}
\newtheorem{cor}[theorem]{Corollary}
\newtheorem{remark}{Remark}
\newtheorem{definition}{Definition}

\usepackage{mathtools}

\makeatletter
\DeclareRobustCommand\widecheck[1]{{\mathpalette\@widecheck{#1}}}
\def\@widecheck#1#2{%
    \setbox\z@\hbox{\m@th$#1#2$}%
    \setbox\tw@\hbox{\m@th$#1%
       \widehat{%
          \vrule\@width\z@\@height\ht\z@
          \vrule\@height\z@\@width\wd\z@}$}%
    \dp\tw@-\ht\z@
    \@tempdima\ht\z@ \advance\@tempdima2\ht\tw@ \divide\@tempdima\thr@@
    \setbox\tw@\hbox{% 
       \raise\@tempdima\hbox{\scalebox{1}[-1]{\lower\@tempdima\box
\tw@}}}%
    {\ooalign{\box\tw@ \cr \box\z@}}}
\makeatother

\newcommand{\ep}{\epsilon}
\newcommand{\om}{\omega}

\DeclareMathOperator{\Res}{Res}

\renewcommand{\Im}{\textrm{Im}}

\newcommand{\ds}{\displaystyle}

\newcommand{\de}{\delta}

\newcommand{\be}{\begin{equation}}
\newcommand{\beq}{\begin{equation}}
\newcommand{\ee}{\end{equation}}
\newcommand{\eeq}{\end{equation}}
\newcommand{\bes}{\begin{equation*}}
\newcommand{\ees}{\end{equation*}}
\newcommand{\mand}{\quad \text{and}\quad}

\newcommand{\R}{{\bf{R}}}
\newcommand{\C}{{\bf{C}}}

\newcommand{\Z}{{\bf{Z}}}

\newcommand{\E}{{\mathcal{E}}}
\renewcommand{\O}{{\mathcal{O}}}

\renewcommand{\tilde}{\widetilde}

\newcommand{\bunderbrace}[2]{%
  \begin{array}[t]{@{}c@{}}
  \underbrace{#1}\\
  #2
  \end{array}
}

\title{Mass-in-Mass Lattices with Small Internal Resonators}
\author{Fazel Hadadifard}
\address{Department of Mathematics, Drexel University, Philadelphia PA. {\it fh352@drexel.edu}}
\author{J. Douglas Wright}
\address{Department of Mathematics, Drexel University, Philadelphia PA. {\it jdw66@drexel.edu}}

\begin{document}
\maketitle
\begin{abstract} We consider 
the mass-in-mass (MiM) lattice when the internal resonators are very small. When
there are no internal resonators the lattice reduces to a standard 
Fermi-Pasta-Ulam-Tsingou (FPUT) system. We show that the solution of the MiM system,
with suitable initial data, shadows the FPUT system for long periods of time. Using 
some classical oscillatory integral estimates we can conclude that the error of the approximation is (in some settings) higher than one may expect.
\end{abstract}


\section{The problem}

We consider the {\it mass-in-mass} (MiM) variant of the {\it Fermi-Pasta-Ulam-Tsingou}
(FPUT) lattice:
infinitely many particles of unit mass (indexed by $j \in \Z$) are arranged on a line, each connected to its nearest neighbors by a ``spring'' with potential energy function $V$ (which we assume  is smooth\footnote{In this paper, when we say ``smooth'' we always mean $C^\infty$.} and has $V(0) =V'(0) = 0<k:= V''(0)$). The displacement of the $j^{\text{th}}$ particle is $U_j$.
 Additionally, each particle is connected by a linear spring (with spring constant $\kappa$) to an internal resonator (of mass $\mu$). The displacement of the $j^{\text{th}}$ resonator is $u_j$.  The equations of motion can be found using Newton's second law:
\be\label{MIM}
\begin{split}
 \ddot{U}_j&= V'(U_{j+ 1}- U_j)- V'(U_j- U_{j-1})+ \kappa (u_j- U_j)\\
 \mu \ddot{u_j}&= \kappa (U_j- u_j).
\end{split}
\ee
These sorts of lattices have been the subject of quite a bit of research of late, in large part because engineers have found a wide variety of applications for apparatus which are modeled by MiM systems.  Applications range from shock absorption \cite{gantzounis} to remote sensing \cite{spadoni} and in areas from medicine \cite{bones} to materials science \cite{cement}.

 Our interest is analytical and in this article we investigate the dynamics of \eqref{MIM} when $0 < \mu \ll 1$, that is when the internal resonators have small mass. 
When $\mu = 0$ the second equation implies $u_j = U_j$ and the first becomes 
\be\label{FPUT}
\ddot{U_j}= V'(U_{j+ 1}- U_j)- V'(U_j- U_{j-1}).
\ee
These are the equations of motion for the standard FPUT. It takes little insight to conjecture that  solutions of \eqref{MIM} shadow solutions of \eqref{FPUT}
when $\mu$ is small. We prove a quantitative version of such a conjecture. However this is not a straightforward result: since $\mu$
 multiplies the highest order derivative in \eqref{MIM}, the problem is one of singular perturbation.  We also find something rather surprising: by slightly adjusting the potential in \eqref{FPUT} and adding some restrictions to the initial conditions for the internal resonators, we can improve the accuracy of the approximation 
 by more than an order of magnitude.

Before getting into the weeds, we make some remarks on a recent spate of articles on MiM and FPUT lattices and how they relate to our work. First we mention the article \cite{KSX} by Kevrekidis, Stefanov \& Xu. The authors use a variational argument to show that for the degenerate Hertzian potential $V_H(h) := h_+^{5/2}$, there exists a countable number of choices for the internal mass $\mu$, converging to zero, for which the MiM system admits
spatially localized traveling wave solutions. 
This work was extended by Faver, Goodman \& Wright in \cite{FGW} to apply to more general, but non-degenerate, potentials. Again, for a sequence of choices of $\mu$ converging to zero, there are spatially localized traveling waves. The argument in \cite{FGW} is perturbative and in particular, uses the $\mu =0$ FPUT traveling wave as the point of bifurcation. In \cite{F}, Faver proves that away from the countable collection of masses, the traveling waves are not spatially localized but instead converge at infinity to very small amplitude periodic waves, {\it i.e.}~nanopterons \cite{boyd}. The point here is that despite the relative simplicity of the system \eqref{MIM}, from the standpoint of traveling wave solutions, the system depends subtly on the mass of the internal resonators. This paper is, in part,
an attempt to address similar issues for the Cauchy problem.
We also mention the article \cite{PS} by Pelinovsky \& Schneider. In that paper the authors treat a diatomic FPUT lattice in the limit that the mass ratio tends to zero. They prove that the small mass ratio lattice is well-approximated by the limiting monatomic FPUT lattice. Their result directly inspired our work here. See Remark \ref{PS remark} for a more thorough comparison of their work and ours.

 \section{First order reformulation and existence of solutions}
 Let
 $$
 R_j := U_{j+1} - U_j, \quad P_j := \dot{U}_j, \quad r_j := u_j - U_j \mand p_j := \dot{u}_j.
 $$
 The variables are (in order): the relative displacement between adjacent external particles; the velocity of the external particles; the relative displacement between the internal resonators and their hosts; the velocity of the internal resonators. In these coordinates \eqref{MIM} reads:
 \begin{equation}\label{FOMIM}
 \begin{split}
 \dot{R} &= \delta^+ P\\
 \dot{P} &= \delta^- [V'(R)] + \kappa r\\
 \dot{r} &= p-P\\
 \mu \dot{p} &= -\kappa r.
 \end{split}
 \end{equation}
We suppress dependence on the lattice site $j$ and use the notation $(\delta^\pm q)_j := \pm(q_{j\pm1}-q_j)$.
In fact \eqref{FOMIM} is in classical hamiltonian form, though since we do not utilize this feature very strongly,
 we do not elaborate.

We view \eqref{FOMIM} as an ODE on the Hilbert space $(\ell^2)^4$. The right hand side can easily be shown to be a smooth map in that topology and thus the Cauchy problem is well-posed by Picard's theorem and solutions exist for at least short periods of time.
In fact solutions exist for all $t$, at least if they are initially not too big.
Before we state the result, we need to define
an appropriate norm for solutions.
Let \be\label{mu norm}
\| (R,P,r,p) \|_\mu :=\sqrt{ {k \over 2} \| R\|^2+
{1 \over 2} \|P\|^2
+ {\kappa \over 2}  \|r\|^2
+ {\mu \over 2} \| p\|^2}.
\ee
Here and throughout we use $$\| \cdot \| := \|\cdot \|_{\ell^2}.$$ 
The norm $\| \cdot \|_\mu$ is just a scaling of the usual  $(\ell^2)^4$ norm and is equal to the (square root of the) mechanical energy of the linearization of \eqref{FOMIM}; recall that $k := V''(0)$.

For a solution $(R,P,r,p)$ of \eqref{FOMIM}, let
$$
H(t):=\sum_{j \in \Z} \left(V(R_j) + {1 \over 2} P_j^2 + {1 \over 2} \kappa r_j^2 + {1 \over 2} \mu p_j^2\right).
$$
If finite at $t =0$, this quantity is constant for all $t$: it is just the mechanical energy of the lattice. 
Here is the calculation:
\bes\begin{split}
\dot{H}(t) 
= &\sum_{j \in \Z} V'(R_j)\dot{R_j} + P_j \dot{P}_j + \kappa r_j \dot{r}_j +\mu p_j \dot{p}_j\\
= &\sum_{j \in \Z} V'(R_j)(\delta^+ P)_j + P_j( (\delta^- [V'(R)])_j + \kappa r_j) + \kappa r_j (p_j - P_j)-\kappa p_j r_j\\
= & 0.
\end{split}\ees
Since $\dot{H}(t) = 0$, $H(t)$ is constant.
In the above we have made liberal use of the summation by parts identity,
$\ds\sum_{j \in \Z} (\delta^+ f)_j g_j = -\sum_{j \in \Z}  f_j {(\delta^- g)}_j.$

The conservation of energy is crucial for proving:
%Since $V''(0)  > 0$, at least for $\| R\|$ not too big,  $\sqrt{H}$ is equivalent
%to $\| \cdot \|_{\mu}$ (with equivalence constants independent of $\mu$).
%Solutions of ODEs obtained from Picard's theorem exist as long as they remain in the domain of definition
%of their vector field, which in this case is all of $(\ell^2)^4$.
%The conservation of $H$ guarantees
%that solutions of \eqref{FOMIM} always have finite norm and thus are always in $(\ell^2)^4$.
%In this way 
%we get global in time solutions of the system \eqref{FOMIM} with control of the norm. 
%Here is the result:
\begin{theorem}\label{exist}
Fix $\kappa>0$ and assume that $V: \R \to \R$ is smooth  with $V(0)~=~ V'(0)~=~0$
and $V''(0) =: k > 0$.
%Let $V''(0) = : k > 0$. Then 
There exists $\rho_*=\rho_*(V)>0,$ such that, for any $\mu > 0$, if 
$$
\|(R_0,P_0,r_0,p_0)\|_\mu \le \rho_*
$$
then the unique solution of the MiM lattice \eqref{FOMIM} with initial data $(R_0,P_0,r_0,p_0)$ exists for all $t \in \R$ and
\bes
\|(R(t),P(t),r(t),p(t))\|_\mu\le 2\|(R_0,P_0,r_0,p_0)\|_\mu.
\ees

\end{theorem}

\begin{proof} The hypotheses on $V$ imply, by way of Taylor's Theorem, the existence of 
$\sigma_*>0$ for which
$%\label{V prop}
| h | \le \sigma_*$ implies $\ds {k \over 4} h^2 \le V(h) \le k h^2.$
So if $\| R \|_{\ell^\infty} \le \sigma_*$ we have
$\ds
\sum_{j \in \Z} {{k \over 4} }R_j^2 \le \sum_{j \in \Z} V(R_j) \le 
\sum_{j \in \Z} k R_j^2.
$
This in turn implies
\be\label{picardequiv1}
{1 \over 2} \|(R,P,r,p)\|^2_\mu
\le H \le 2 \|(R,P,r,p)\|^2_\mu
\ee
when \be\label{condition}
\|R\|_{\ell^\infty} \le \sigma_*.
\ee
That is to say when \eqref{condition} holds, $\sqrt{H}$ and $\|(R,P,r,p)\|_\mu$ are equivalent.

Since $H$ is constant, \eqref{picardequiv1} gives us:
\bes
{1 \over 2} \|(R(t),P(t),r(t),p(t))\|^2_\mu
\le H(t) = H(0)  \le 2 \|(R_0,P_0,r_0,p_0)\|^2_\mu.
\ees
If we cut out the middle terms and do some simple algebra we arrive at
\be\label{goal}
\|(R(t),P(t),r(t),p(t))\|_\mu \le 2 \|(R_0,P_0,r_0,p_0)\|_\mu.
\ee
This is the final estimate in the theorem but we are not yet done. The reason is that
\eqref{goal} only holds for those values of $t$ where \eqref{condition} is true.

By restricting the initial data, we can ensure that \eqref{condition} holds for all $t$
and thus so does \eqref{goal}. Here is the argument.
We have the ``$\ell^2 \subset \ell^\infty$ embedding estimate''
$\| R\|_{\ell^\infty} \le \| R \|$.
Moreover, the definition of $\|(R,P,r,p) \|_\mu$ implies
$\ds 
\|R\| \le \sqrt{2/k} \|(R,P,r,p)\|_\mu.
$
Putting these together with \eqref{condition} we see that we have \eqref{goal} for those
$t$ when 
\be\label{better condition}
 \|(R(t),P(t),r(t),p(t))\|_\mu \le \sqrt{k \over 2}\sigma_*.
\ee

Now assume
\be\label{rhostar}
 \|(R_0,P_0,r_0,p_0)\|_\mu \le  {1 \over 4}\sqrt{ k \over 2} \sigma_*=:\rho_*. 
\ee
Thus \eqref{better condition} holds initially and the inequality is
strict.
The solution of \eqref{FOMIM} with 
this initial data either satisfies \eqref{better condition} for all $t \in \R$ (in which case we have
\eqref{goal} for all $t \in \R$ and we are done) or it does not. 

If it does not then, because
the solution is continuous in $t$, there is a time $t_1$ for which 
\be\label{doom}
\|(R(t_1),P(t_1),r(t_1),p(t_1))\|_\mu = \sqrt{k \over 2}\sigma_*.\ee
But note that at this time
\eqref{better condition} is met and so we have \eqref{goal}.
Putting \eqref{doom}, \eqref{goal} and \eqref{rhostar} together we obtain
$$
\sqrt{k \over 2}\sigma_*=\|(R(t_1),P(t_1),r(t_1),p(t_1))\|_\mu 
\le 2  \|(R_0,P_0,r_0,p_0)\|_\mu \le{1 \over 2} \sqrt{k \over 2}\sigma_*.
$$
This is an absurdity and thus \eqref{better condition} is met for all $t$ and we are done.

\end{proof}

 \section{The approximation theorem}
 In this section we prove a general approximation theorem for \eqref{FOMIM}.
 Once this is done, we will turn our attention to the specific problem of approximating MiM by FPUT.
 
 For any function
 $$
 \tilde{\Phi}_j(t)= (\tilde{R}_j(t),\tilde{P}_j(t),\tilde{r}_j(t),\tilde{p}_j(t))
$$
define the {\it residuals} 
\be\label{res des}\begin{split}
\Res_1(\tilde{\Phi}) &:= \de^+ \tilde{P} - \dot{\tilde{R}}\\
\Res_2(\tilde{\Phi}) &:= \de^-[V'(\tilde{R})] + \kappa r- \dot{\tilde{P}}\\
\Res_3(\tilde{\Phi}) &:= \tilde{p}  - \tilde{P} - \dot{\tilde{r}}\\
\Res_4(\tilde{\Phi}) &:= -\kappa\tilde{r} - \mu \dot{\tilde{p}}.
\end{split}\ee
The residuals are identically zero if and only if
$\tilde{\Phi}$ solves \eqref{FOMIM}. Our result
gives sufficient conditions on $\tilde{\Phi}$ so that the smallness of the
residuals implies solutions of \eqref{FOMIM} are well-approximated by $\tilde{\Phi}$.

\begin{definition}\label{hyp}
We say $
\left\{ \tilde{\Phi}^\mu= (\tilde{R}^\mu,\tilde{P}^\mu,\tilde{r}^\mu,\tilde{p}^\mu)
\right \}_{\mu \in (0,\mu_0] }
$
is {\bf a family of good approximators of $\O(\mu^N)$ for \eqref{FOMIM} on the interval
$[-T_*,T_*]$} if the following occur.

First,
$$
\left\{ \tilde{\Phi}^\mu \right\}_{\mu \in(0,\mu_0]}
\subset C^1([-T_*,T_*];(\ell^2)^4).
$$

Second, the residuals are small: there exists $C_0>0$
for which $\mu \in (0,\mu_0]$ implies
\be\label{res is small}\tag{D1}
\sup_{|t|\le T_*}\sqrt{\|\Res^{\mu}_1(\tilde{\Phi}^\mu) \|^2+
\|\Res^{\mu}_2(\tilde{\Phi}^\mu)\|^2+
\|\Res^{\mu}_3(\tilde{\Phi}^\mu)\|^2+
{1 \over {\mu}} \|\Res^{\mu}_4(\tilde{\Phi}^\mu)\|^2 } \le C_0 \mu^{N}.
\ee

Lastly, $\tilde{R}^\mu$ and $\partial_t \tilde{R}^\mu$ are not too big: there exist $\alpha_*$, $\beta_*>0$
so that
$\mu \in (0,\mu_0]$ implies
\be\label{R is not big}\tag{D2}
\sup_{|t|\le T_*} \| \tilde{R}^\mu\|_{\ell^\infty} \le \alpha_* \mand \sup_{|t|\le T_*} \left\| \partial_t \tilde{R}^\mu\right\|_{\ell^\infty}\le \beta_*.
\ee
We additionally require that 
\be\label{alpha condition}\tag{D3}
\alpha_* \le
\sup\left\{ \alpha: V''([-\alpha,\alpha]) \subset [k/2,2k] \right\}.
\ee

\end{definition}

Here is our result:
\begin{theorem}
\label{gen approx} Fix $\kappa>0$ and assume that $V: \R \to \R$ is smooth with $V(0)~=~ V'(0)~=~0$
and $V''(0) =: k > 0$. Suppose that $
\left\{ \tilde{\Phi}^\mu= (\tilde{R}^\mu,\tilde{P}^\mu,\tilde{r}^\mu,\tilde{p}^\mu)
\right \}_{\mu \in (0,\mu_0] }
$
is a family of good approximators of $\O(\mu^N)$ for \eqref{FOMIM} on the interval
$[-T_*,T_*]$,  where $N>0$.

Then, for all $K_*>0$, there exists positive constants $\mu_*$ and $C_*$ such
that the following holds when $\mu \in (0,\mu_*]$.
If
\be\label{ic}
\| \Phi^\mu_0 - \tilde{\Phi}^\mu(0)\|_\mu
%\| R^\mu_0 - \tilde{R}^\mu(0)\| +\| P^\mu_0 - \tilde{P}^\mu(0)\| +
%\| r^\mu_0 - \tilde{r}^\mu(0)\| +
%\sqrt{\mu} \| p^\mu_0 - \tilde{p}^\mu(0)\| 
 \le K_* \mu^N
\ee
and $\Phi^\mu$ %= (R^\mu,P^\mu,r^\mu,p^\mu)$
 is the solution of \eqref{FOMIM} with initial data $\Phi^\mu_0$ then
\be\label{approx}
\| \Phi^\mu(t) - \tilde{\Phi}^\mu(t)\|_\mu
 \le C_* \mu^N
\ee
for all $t \in [-T_*,T_*]$.

That is to say, if $\Phi^\mu$ and $\tilde{\Phi}^\mu$ are initially $\O(\mu^N)$ close
then they are $\O(\mu^N)$ close on all of $[-T_*,T_*]$.

\end{theorem}

\begin{proof} 

{\it Part 1---the Error Equations:}
Let $$\Psi=(\psi_1,\psi_2,\psi_3,\psi_4):=\Phi^\mu - \tilde{\Phi}^\mu.$$
This is the error between the true solution and the approximator. 
A direct calculation shows that $\Psi$ satisfies
\be\label{Psi eqn} \begin{split}
\dot{\psi}_1 &= \de^+ \psi_2 + \Res^\mu_1(\tilde{\Phi}^\mu) \\
\dot{\psi}_2& = \de^-\left[W'(\psi_1;t) \right] + \kappa \psi_3 + \Res^\mu_2(\tilde{\Phi}^\mu)\\
\dot{\psi}_3& = \psi_4 - \psi_2+\Res^\mu_3(\tilde{\Phi}^\mu)\\
\mu\dot{\psi}_4& = - \kappa \psi_3 + \Res^\mu_4(\tilde{\Phi}^\mu)
\end{split}\ee
where
$$
W'_j(\zeta;t):=V'(\tilde{R}^\mu_j(t)+\zeta)- V'({\tilde{R}^\mu_j(t)}).
$$
Note that
$W'_j(\zeta;t) = 
 \partial_\zeta W_j(\zeta;t)$
with
\be\label{W}
W_j(\zeta;t):= V(\tilde{R}^\mu_j(t)+\zeta)- V({\tilde{R}^\mu_j(t)}) - V'(\tilde{R}^\mu_j(t)) \zeta.
\ee
We are done when we show that $\|\Psi(t) \|_\mu \le C_* \mu^N$ for $t\in [-T_*,T_*]$.

{\it Part 2---the Modified Energy:} The heart of the proof is closely related to the conservation of the energy $H$.
Let
$$
E(t):=  \sum_{j \in \Z} \left(W(\psi_1;t)+ {1\over 2} \psi^2_2 + {1 \over 2} \kappa \psi_3^2 + {1 \over 2} \mu \psi_4^2\right).
$$
This quantity is a modification of $H$ and, while it is not conserved, grows only slowly.
Below, we will show that $\sqrt{E}$  is equivalent to $\| \Psi \|_\mu$, but first
we compute its time derivative in order to develop the key energy estimate:
$$
\dot{E}(t) = \sum_{j\in\Z} \left(W'(\psi_1) \dot{\psi}_1 + \partial_t W(\psi_1;t) + \psi_2 \dot{\psi}_2+ \kappa \psi_3 \dot{\psi}_3+  \mu \psi_4 \dot{\psi}_4\right).
$$
Using \eqref{Psi eqn}
\bes
\begin{split} \dot{E}(t) = 
\sum_{j \in \Z} 
&\bigg( 
W'(\psi_1;t) \left( \de^+ \psi_2 + \Res^\mu_1(\tilde{\Phi}^\mu)\right)
+
\psi_2 \left(  \de^-\left[W'(\psi_1;t) \right] + \kappa \psi_3 + \Res^\mu_2(\tilde{\Phi}^\mu)\right)
\\
&+
\kappa \psi_3\left( \psi_4 - \psi_2+\Res^\mu_3(\tilde{\Phi}^\mu)\right)
+
\psi_4 \left(- \kappa \psi_3 + \Res^\mu_4(\tilde{\Phi}^\mu)\right) 
+
\partial_t W(\psi_1;t)  \bigg).
\end{split}
\ees
There are many cancelations:
\bes
 \dot{E}(t) = 
\sum_{j \in \Z} 
\bigg( 
W'(\psi_1;t) \Res^\mu_1(\tilde{\Phi}^\mu)
+
\psi_2 \Res^\mu_2(\tilde{\Phi}^\mu)
+ 
\kappa \psi_3 \Res^\mu_3(\tilde{\Phi}^\mu)
+ 
\psi_4 \Res^\mu_4(\tilde{\Phi}^\mu)
+
\partial_t W(\psi_1;t)  \bigg).
\ees

%{\it Part 3---Elementary Estimates:}
%Now we begin estimating $\dot{E}$.
%First Cauchy-Schwarz:
%\bes\begin{split}
% \dot{E}(t) \le&
%\|W'(\psi_1;t)\|\| \Res^\mu_1(\tilde{\Phi}^\mu)\|
%+
%\|\psi_2 \| \| \Res^\mu_2(\tilde{\Phi}^\mu)\|\\
%+& 
%\kappa \|\psi_3\| \| \Res^\mu_3(\tilde{\Phi}^\mu)\|
%+ 
%\|\psi_4\| \|\Res^\mu_4(\tilde{\Phi}^\mu)\|
%+
%\|\partial_t W(\psi_1;t)\|_{\ell^1}.
%\end{split}\ees
%Next is Young's inequality:
%\bes\begin{split}
% \dot{E}(t) \le
%&{1 \over 2} \|W'(\psi_1;t)\|^2 
%+
%{1 \over 2}\|\psi_2 \|^2
%+ 
%{\kappa^2  \over 2} \|\psi_3\|^2 
%+ 
%{\mu \over 2} \|\psi_4\|^2 +
%\|\partial_t W(\psi_1;t)\|_{\ell^1}\\
%+&{1 \over 2} \|\Res^{\mu}_1(\tilde{\Phi}^\mu) \|^2+
%{1 \over 2} \|\Res^{\mu}_2(\tilde{\Phi}^\mu)\|^2+
%{1 \over 2} \|\Res^{\mu}_3(\tilde{\Phi}^\mu)\|^2+
%{1 \over 2{\mu}} \|\Res^{\mu}_4(\tilde{\Phi}^\mu)\|^2.
%\end{split}\ees
Using the Cauchy-Schwarz inequality, Young's inequality and 
 \eqref{res is small} we estimate the above:
\be\label{before}
\dot{E}(t) \le {1 \over 2} \|W'(\psi_1;t)\|^2 
+
{1 \over 2}\|\psi_2 \|^2
+ 
{\kappa^2  \over 2} \|\psi_3\|^2 
+ 
{\mu \over 2} \|\psi_4\|^2 + \|\partial_t W(\psi_1;t)\|_{\ell^1} + {1 \over 2} C_0^2 \mu^{2N}.
\ee
To go further than this, we need more information about $W$.

{\it Part 3---Estimates for $W$:}
Taylor's theorem tells us that for $\zeta \in \R$  we have
$$
W_j(\zeta;t) = {1 \over 2}  V''(z_j(t)) \zeta^2
$$
where $z_j(t)$ lies between $\tilde{R}^\mu_j(t)$ and $\tilde{R}^\mu_j(t)+ \zeta$.
We have assumed \eqref{R is not big} and the condition \eqref{alpha condition} on $\alpha_*$
tells us that $V''(\tilde{R}^\mu_j(t)) \in [k/2,2k]$
for $j \in \Z$, $t\in[-T_*,T_*]$ and $\mu \in (0,\mu_0]$.
Thus, since $V$ is smooth, there exists $\tau_*>0$ so that $|\zeta| \le \tau_*$ implies
$V''(z_j(t)) \in [k/4,4k]$ and as such
\be\label{yoohoo}
{k \over 8}  \zeta^2 \le W_j(\zeta;t) \le 2 k \zeta^2.
\ee
%The above holds for all $j \in \Z$, $t \in [-T_*,T_*]$ and $\mu \in (0,\mu_0]$.

Now suppose that $\gamma \in \ell^2$ has $\|\gamma\| \le \tau_*$. Since $\ell^2 \subset \ell^\infty$ we have $\|\gamma\|_{\ell^\infty} \le \|\gamma\|$. 
Thus \eqref{yoohoo} gives us:
$$\ds
{k \over 8}  \gamma_j^2 \le W_j(\gamma_j;t) \le 2 k \gamma_j^2.
 $$
And so
\be\label{equiv heart}
\| \gamma\| \le \tau_* \implies
{k \over 8} \|\gamma \|^2 \le \sum_{j \in \Z} W_j(\gamma_j;t) \le {2k} \|\gamma \|^2.
\ee
This estimate in turn implies that, for all $t \in [-T_*,T_*]$ and $\mu \in (0,\mu_0]$,
\be\label{true}
\| \psi_1\| \le \tau_* \implies
{1 \over 4} \| \Psi \|_\mu^2 \le E(t) \le 4 \| \Psi \|_\mu^2.
\ee
This is the equivalence of $\sqrt{E}$ and $\| \Psi\|_\mu$ which was foretold.
Completely analogous calculations can be used to show that 
%for all $t \in [-T_*,T_*]$ and $\mu \in (0,\mu_0]$,
\be\label{eq jr}
\| \gamma\| \le \tau_* \implies
\| W'(\gamma;t)\| \le  4k \|\gamma \|.
%\mand 
%\| \partial_t W(\gamma;t)\|_{\ell^1} \le  \beta_1 \|\gamma \|^2.
\ee
%The constant $\beta_1>0$ does not depend on $\mu$.

We also need an estimate on $\partial_t W$. Computing the derivative gets:
$$
\partial_t W_j(\zeta;t)=\left[
V'(\tilde{R}_j^\mu(t) + \zeta)
-V'(\tilde{R}_j^\mu(t))
-V''(\tilde{R}_j^\mu(t))\zeta
\right] \partial_t{\tilde{R}}_j^\mu.
$$
Taylor's theorem tells us that
$$
\partial_t W_j(\zeta;t) = {1 \over 2} V'''(z_j(t)) \zeta^2
\partial_t{\tilde{R}}_j^\mu
$$
with $z_j(t)$ in between $\tilde{R}_j^\mu$ and $\tilde{R}^\mu_j + \zeta$.
Letting $\ds \beta_0:=\max_{|\rho|\le \tau_*+\alpha_*} |V'''(\rho)|$ and using 
the estimate for $\partial_t{\tilde{R}}_j^\mu$ in \eqref{R is not big} we now see that
$$
|\partial_t W_j(\zeta;t)| \le {1 \over 2} \beta_0 \beta_* \zeta^2
$$
when $|\zeta| \le \tau_*$. Thus we find that for all $t \in [-T_*,T_*]$ and $\mu \in (0,\mu_0]$
\be\label{wt}
\| \gamma\| \le \tau_* \implies
\| \partial_t W(\gamma;t)\|_{\ell^2} \le  \beta_2 \|\gamma \|^2
\ee
where $\beta_2 := \beta_0 \beta_*/2$.

{\it Part 4---Final Steps:} Applying \eqref{true}, \eqref{eq jr} and \eqref{wt} to \eqref{before}
gets us
$$
\dot{E} \le \Gamma_* \left( E+  \mu^{2N}\right)
$$
so long as $\|\psi_1\| \le \tau_*$. The constant $\Gamma_*=\Gamma_*(V,\beta_*,\kappa,C_0)>0$ is
independent of $\mu$.

We apply Gr\"onwall's inequality and get
$$
E(t) \le e^{\Gamma_* t} \left(E(0) +\mu^{2N} \right).
$$
Then we use \eqref{true} again:
$$
\|\Psi(t)\|^2_\mu \le 16 e^{\Gamma_* t} \left(\|\Psi(0)\|^2_\mu +\mu^{2N} \right).
$$

We have assumed that $\|\Psi(0)\|_\mu \le K_* \mu^{N}$ 
and we know $|t| \le T_*$ so we have
$$
\|\Psi(t)\|_\mu \le \bunderbrace{4 e^{\Gamma_* T_*/2} \sqrt{K_*^2+1}}{C_*}\mu^{N}.
$$
The constant $C_*$ does not depend on $\mu$, but the above estimate
holds only so long as $\|\psi_1\| \le \tau_*$. But we can make the right hand side of this last displayed inequality (which controls $\|\psi_1\|$) as small as we like, so this restriction is not a serious one. And so we find that there exists $\mu_* > 0$ so that $\mu \in (0,\mu_*]$ implies $\|\Psi(t)\| \le C_*\mu^N$
for all $|t| \le T_*$ and we are done with the proof.

\end{proof}

\section{The leading order FPUT approximation}\label{FPUT approx}

In \eqref{FOMIM}, if we put $\mu = 0$ we find that the last two equations become:
\begin{equation}\label{subsys 2}
r= 0 \quad \text{and} \quad p = P.
\end{equation}
That is to say, as one may expect, the internal resonators are fixed at the center of their hosting particle and their velocity $p$ is exactly equal to that of its host.
Then we put \eqref{subsys 2} into the first two equations of \eqref{FOMIM}:
\begin{equation}\label{subsystem}
\dot{R}= \de^+ P \quad \text{and} \quad
\dot{P}=  \de^- [V'(R)].
\end{equation}
Of course \eqref{subsystem} is just a vanilla monatomic FPUT lattice, equivalent to \eqref{FPUT}.
So our approximating system is 
\be
\tilde{\Phi}_{FPUT}:=(\tilde{R},\tilde{P},0,\tilde{P})
\ee
where $(\tilde{R},\tilde{P})$ solves \eqref{subsystem}. 

Now we will show that $\tilde{\Phi}_{FPUT}$ is a good approximator; note that it does not depend on $\mu$, though the residuals will.
An argument identical to that which led to Theorem \ref{exist} tells us that there is a positive constant $\rho_1$,  such that $\|\tilde{R}(0)\| + \|\tilde{P}(0)\| \le \rho_1$ implies
\be\label{FPUT control}\|\tilde{R}(t)\| + \|\tilde{P}(t)\| \le 2 \left( \|\tilde{R}(0)\| + \|\tilde{P}(0)\|\right)\ee for all $t \in \R$. 
Thus, so long as $\|\tilde{R}(0)\| + \|\tilde{P}(0)\|$ is not too big,  the conditions \eqref{R is not big} and \eqref{alpha condition} are more or less automatically met and, moreover, they hold for all $t \in \R$.

We compute directly that
$$
\Res_1(\tilde{\Phi}_{FPUT}) = 
\Res_2(\tilde{\Phi}_{FPUT}) = 
\Res_3(\tilde{\Phi}_{FPUT}) = 0
$$
and
$$
\Res_4(\tilde{\Phi}_{FPUT}) = -\mu \dot{\tilde{P}} = -\mu \de^- [V'(\tilde{R})].
$$
Thus 
\begin{multline*}
\sqrt{\|\Res^{\mu}_1(\tilde{\Phi}_{FPUT}) \|^2+
\|\Res^{\mu}_2(\tilde{\Phi}_{FPUT})\|^2+
\|\Res^{\mu}_3(\tilde{\Phi}_{FPUT})\|^2+
{1 \over {\mu}} \|\Res^{\mu}_4(\tilde{\Phi}_{FPUT})\|^2 } \\
= \sqrt{\mu}  \|\de^- [V'(\tilde{R})]\|.
\end{multline*}
Standard estimates and \eqref{FPUT control} tell us that 
$\sqrt{\mu}  \|\de^- [V'(\tilde{R})]\| \le C_0\sqrt{\mu}$ for all $t\in \R$.
So we have \eqref{res is small} with $N = 1/2$.
We now call on Theorem \ref{gen approx} and get: 

\begin{cor} \label{naive}
Let $\kappa>0$, $K_*>0$, $T_*>0$ and $V: \R \to \R$ be smooth with $V(0) = V'(0) = 0$
and $V''(0) =: k > 0$.  Then there exist
$\rho_* = \rho_*(V)>0$, $\mu_*=\mu_*(K_*,T_*,\kappa,V)>0$ and $C_* = C_*(K_*,T_*,\kappa,V)>0$ 
for which
we have the following when $\mu \in (0,\mu_*]$.

Suppose that $(\tilde{R},\tilde{P})$ solves the FPUT system \eqref{subsystem} with 
$$
\|\tilde{R}(0)\| + \|\tilde{P}(0)\| \le \rho_*$$ and
$(R,P,r,p)$ solves the MiM lattice \eqref{FOMIM} with 
$$
\|(R(0),P(0),r(0),p(0)) - (\tilde{R}(0),\tilde{P}(0),0,\tilde{P}(0))\|_\mu \le K_* \sqrt{\mu}.
$$
Then
$$
\|(R(t),P(t),r(t),p(t)) - (\tilde{R}(t),\tilde{P}(t),0,\tilde{P}(t))\|_\mu \le C_* \sqrt{\mu}
$$
for all $t \in [-T_*,T_*]$.
\end{cor}

\begin{remark}\label{PS remark}
As we mentioned in the introduction, the article \cite{PS} treats the monatomic limit
of a diatomic FPUT lattice in the case of small mass ratio. Their mass ratio 
is named $\ep^2$ and is most comparable to our internal mass $\mu$. 
Their main result, Theorem 1, gives a rigorous error bound of $\O(\ep)$ on $\O(1)$
time scales. Given the comparison $\ep^2 \sim \mu$, our result here is exactly the analogous
one for MiM with small internal resonators. 

\end{remark}

\section{Higher order expansions}
The final two equations in \eqref{FOMIM} are solvable for $(r,p)$ in terms of $(R,P)$ with elementary ODE techniques. In this way we can eliminate $(r,p)$ from the system (almost)
entirely and are left with what is a perturbation of FPUT with a continuous delay term. 
This delay term can then be  approximated using classical oscillatory integral
methods.
Then we will use Theorem \ref{gen approx} to justify some of these approximations,
which are of a higher order in $\mu$ than what we saw in Corollary \ref{naive}.

\subsection{Delay equation reformulation}
Take the time derivative of the equation for $\dot{r}$ in \eqref{FOMIM} and get
\be\label{SHO}
\ddot{r} = -\omega_\mu^2r -\dot{P}
\ee
%Next we replace $\dot{P}$  with the right hand side of the second equation in \eqref{FOMIM} to get:
%\be\label{SHO}
%\ddot{r} = -\omega_\mu^2 r -\de^-[V'(R)] 
%\ee
%where
%$$
%\omega_\mu^2:={\kappa(1+\mu) \over \mu} .
%$$
%The equation \eqref{SHO} can be solved with the variation of parameters formula:
%\bes\begin{split}
%r(t) &={r(0) \cos(\omega_\mu t) + \omega_\mu^{-1} (p(0) - P(0)) \sin(\omega_\mu t) 
%- \omega_\mu^{-1} \int_0^t \sin(\omega_\mu(t-t')) \de^-[V'(R(t'))]dt'}\\
%&=:{F^\mu[r(0),p(0),R]}.
%\end{split}\ees
%The equation for $\dot{r}$ can be used to figure out $p$:
%\bes\begin{split}
%p(t) &= {P(t) -\omega_\mu r(0) \sin(\omega_\mu t) + 
%(p(0) - P(0)) \cos(\omega_\mu t) - \int_0^t \cos(\omega_\mu(t-t')) \de^-[V'(R(t'))] dt'}\\
%&=:{G^\mu[r(0),p(0),R,P]}.
%\end{split}\ees
where $$\omega_\mu := \sqrt{\kappa/\mu}.$$
We solve \eqref{SHO} using variation of parameters:
%$$
%\ddot{r} = -\omega_\mu^2 r -\dot{P}
%$$
%The variation of parameters formula then gives:
\bes\begin{split}
r_j(t) = &\bunderbrace{\left[r_j(0) \cos(\omega_\mu t) + {1 \over \omega_\mu} (p_j(0) - P_j(0)) \sin(\omega_\mu t) \right]
- {1 \over \omega_\mu}  \int_0^t \sin(\omega_\mu(t-t')) \dot{P}_j(t') dt'.}{
F^\mu[r(0),p(0),P]}
\end{split}\ees
Though we do not use it, the equation for $\dot{r}$ can be used to figure out $p$:
\bes\begin{split}
p_j(t) = &{\left[P_j(t) -\omega_\mu r_j(0) \sin(\omega_\mu t) + 
(p_j(0) - P_j(0)) \cos(\omega_\mu t)\right] - \int_0^t \cos(\omega_\mu(t-t')) \dot{P}(t') dt'.}
%{G^\mu[r(0),p(0),P]}
\end{split}\ees
Putting the solution for $r$ back in the first two equations of \eqref{FOMIM} gets:
 \begin{equation}\label{FOMIM2}
 \begin{split}
 \dot{R} &= \delta^+ P\\
 \dot{P} &= \delta^- [V'(R)] + \kappa F^\mu[r(0),p(0),P]. %\dot{r} &= p-P\\
 %\mu \dot{p} &= -\kappa r.
 \end{split}
 \end{equation}
This system is equivalent to \eqref{FOMIM};
only the initial conditions of $(r,p)$ still play a role.  Because of the integral in $F^\mu$, 
this is a continuous delay equation. 

\subsection{The general strategy}\label{general strategy}
%Because $\sin(\omega_\mu(t-t'))$ is present in the integral
%term in $F^\mu$, and because $\omega_\mu \to \infty$ as $\mu \to 0^+$, we will
%can use oscillatory integral techniques to approximate it.

Suppose we have an approximation of $F^\mu$:
 $$
 F^\mu[r(0),p(0),P] = \tilde{F}^\mu + \O(\mu^N).
 $$
 Then we can make  an approximating system easily:
  \begin{equation}\label{FOMIMapprox}
 \begin{split}
 \dot{\tilde{R}} &= \delta^+ \tilde{P}\\
 \dot{\tilde{P}} &= \delta^- [V'(\tilde{R})] + \kappa \tilde{F}^\mu\\
 \dot{\tilde{r}} &= \tilde{p}-\tilde{P}\\
 \mu \dot{\tilde{p}} &= -\kappa \tilde{r}.
 \end{split}
 \end{equation}
For this approximating  system we have
 $$
 \Res^\mu_1(\tilde{\Phi}^\mu) = 
  \Res^\mu_3(\tilde{\Phi}^\mu) = \Res^\mu_4(\tilde{\Phi}^\mu)=0
 $$
 and
 $$
 \Res^\mu_2(\tilde{\Phi}^\mu) = \kappa F^\mu[\tilde{r}(0),\tilde{p}(0),\tilde{P}] - \kappa \tilde{F}^\mu.
 $$
 Thus, modulo some details, Theorem \ref{gen approx} tells us that the error made by this approximation is $\O(\mu^N)$. 
 The point here is that now all we have to do is find expansions of $F^\mu$. 
 Note that doing so does imply additional conditions on the initial data.
 
% \subsection{FPUT approximation revisited}
% To illustrate the approach, we make 
%the most naive approximation of $F^\mu$ and put

\subsection{Oscillatory integral expansions}
We put
$$
F^\mu[r(0),p(0),P] = \left[ r(0) \cos(\omega_\mu t) +{1 \over \omega_\mu}\left(p(0) - P(0)\right) \sin(\omega_\mu t)\right]+I_\mu[\dot{P}](t)
$$
where
$$
I^\mu[Q](t):= -{1 \over \omega_\mu}\Im \int_0^t e^{i\omega_\mu(t-t')}Q(t') dt'.
$$
Since $\omega_\mu = \sqrt{\kappa/\mu}$, the frequency of the complex sinusoid
is very high as $\mu \to 0^+$ and we can use classical oscillatory integral techniques to expand $I^\mu$ in (negative) powers of $\om_\mu$.
Specifically, we use the following lemma whose proof (which we omit) is
obtained by integrating by parts many, {\it many} times:
\begin{lemma} Suppose that $f(t)$ is $C^{n+1}(\R,\C)$ and $\om \ne 0$. Then
\be\label{osc}
\begin{split}
\int_0^t e^{i \omega(t-t')} f(t') dt' = &{i \over \omega} \sum_{j=0}^n \left(-{i \over \omega}\right)^j f^{(j)}(t) -
 {ie^{i \omega t} \over \omega} \sum_{j=0}^n \left(-{i \over \omega}\right)^j f^{(j)}(0) \\+&\left(-{i \over \omega}\right)^{n+1} \int_0^t e^{i \omega(t-t')} f^{(n+1)}(t') dt'.
\end{split}
\ee
\end{lemma}
In this lemma, the integral term and the $j  = n$ terms in the sums are $\O(1/\om^{n+1})$
and all other terms are lower order. Using this observation
 we get the expansion
\be\label{I expand 0}
I_\mu[Q](t) = -\Im
\left( {i \over \omega_\mu^2} \sum_{j=0}^{n-1} \left(-{i \over \omega_\mu}\right)^j Q^{(j)}(t) -
 {ie^{i \omega_\mu t} \over \omega_\mu^2} \sum_{j=0}^{n-1} \left(-{i \over \omega_\mu}\right)^j Q^{(j)}(0)\right) + \mathcal{E}_n^\mu[Q](t)
\ee
where the estimate
\be\label{E est}
\| \mathcal{E}_n^\mu[Q](t)\| \le {C\over \omega_\mu^{n+2}}\left( \| Q^{(n)}(t)\| +\| Q^{(n)}(0)\|+|t|\sup_{|t'|\le |t|}\| Q^{(n+1)}(t')\| \right)
\ee
is easily obtained. The above estimate tells us that we expect $\E_n^\mu = \O(\mu^{n/2+1})$. %Note that for $n = 0$, we interpret the sums in \eqref{I expand 0} as being zero.

If $Q$  is purely real (as in our application), taking the imaginary part eliminates the 
odd values of $j$ from the first sum in the expansion of $I_\mu$. This, and 
the annoying but easily verified fact that
$$
\Im (i e^{ i \omega t} (-i)^j)
%= (-1)^j \Im( i^{j+1} \cos(\om t) + i^{j+2} \sin(\om t) )
= \begin{cases}
 (-1)^{j/2}\cos(\om t), & \text{$j$ is even}\\
 (-1)^{(j-1)/2}\sin(\om t), & \text{$j$ is odd}
\end{cases}
$$
lead us to:
\bes\begin{split}
I_\mu[Q](t) =& -
{1 \over \omega_\mu^2} \sum_{j=0,even}^{n-1} {(-1)^{j/2}\over \omega^j_\mu} Q^{(j)}(t) \\
&+{1 \over \omega_\mu^2} \left(\sum_{j=0,even}^{n-1} { (-1)^{j/2}\over \omega_\mu^{j}}
Q^{(j)}(0)\right)\cos(\om_\mu t)\\
& +{1 \over \omega_\mu^2}\left( \sum_{j=1,odd}^{n-1} {(-1)^{(j-1)/2} \over \omega_\mu^{j}}Q^{(j)}(0)\right)\sin(\om_\mu t)\\
 &+ \mathcal{E}_n^\mu[Q](t).
\end{split}\ees

The first sum is over evens and so only changes for every other $n$. To squeeze the most out of the above expansion we therefore choose 
$n= 2m$ for integers $m$. A bit of reindexing gives us:
%\bes\begin{split}
%I_\mu[Q](t) =& -
%{1 \over \omega_\mu^2} \sum_{j=0,even}^{2m-1} {(-1)^{j/2}\over \omega^j_\mu} Q^{(j)}(t) \\
%&+{1 \over \omega_\mu^2} \sum_{j=0,even}^{2m-1} { (-1)^{j/2}\over \omega_\mu^{j}}
%\cos(\om t)Q^{(j)}(0)\\
%& +{1 \over \omega_\mu^2} \sum_{j=1,odd}^{2m-1} {(-1)^{(j-1)/2} \over \omega_\mu^{j}}\sin(\om t)Q^{(j)}(0)\\
% &+ \mathcal{E}_{2m}^\mu[Q](t).
%\end{split}\ees
%
%Because stuff is even:
%\bes\begin{split}
%I_\mu[Q](t) =& -
%{1 \over \omega_\mu^2} \sum_{j=0,even}^{2m-2} {(-1)^{j/2}\over \omega^j_\mu} Q^{(j)}(t) \\
%&+{1 \over \omega_\mu^2} \sum_{j=0,even}^{2m-2} { (-1)^{j/2}\over \omega_\mu^{j}}
%\cos(\om t)Q^{(j)}(0)\\
%& +{1 \over \omega_\mu^2} \sum_{j=1,odd}^{2m-1} {(-1)^{(j-1)/2} \over \omega_\mu^{j}}\sin(\om t)Q^{(j)}(0)\\
% &+ \mathcal{E}_{2m}^\mu[Q](t).
%\end{split}\ees
%
%Reindex: ($j = 2k$ or $j= 2k +1$)
\be\label{I expand}\begin{split}
I_\mu[Q](t) =& -
{1 \over \omega_\mu^2} \sum_{k=0}^{m-1} {(-1)^{k}\over \omega^{2k}_\mu} Q^{(2k)}(t) \\
&+{1 \over \omega_\mu^2} \left(\sum_{k=0}^{m-1} { (-1)^{k}\over \omega_\mu^{2k}}
Q^{(2k)}(0)\right)\cos(\om_\mu t)\\
& +{1 \over \omega_\mu^3} \left(\sum_{k=0}^{m-1} {(-1)^{k} \over \omega_\mu^{2k}}Q^{(2k+1)}(0)\right)\sin(\om_\mu t)\\
 &+ \mathcal{E}_{2m}^\mu[Q](t).
\end{split}\ee

\subsection{The FPUT approximation revisited}
Now that we have our oscillatory integral expansions \eqref{I expand}, we get back
to approximating solutions of \eqref{FOMIM}.
Applying \eqref{I expand} with $m = 0$ to 
 $F^\mu[r(0),p(0),P]$ yields
 \be\label{F0}\begin{split}
F^\mu[r(0),p(0),P] =  
&\left[r(0) \cos(\omega_\mu t) + {1 \over \omega_\mu} \left( p(0) - P(0)
\right) \sin(\omega_\mu t) \right]+ \mathcal{E}_{0}^\mu[\dot{P}](t).
\end{split}\ee
Our computations above indicate that $\E_0^\mu$ is $\O(\mu)$ and we 
can make the other terms above small by restrictions on the initial conditions.
So we put
$$
\tilde{F}^\mu=0.
$$
In which case the approximating system \eqref{FOMIMapprox} consists of a standard FPUT
\begin{equation}\label{FOMIMapprox0FPUT}
 \begin{split}
 \dot{\tilde{R}} &= \delta^+ \tilde{P}\\
 \dot{\tilde{P}} &= \delta^- [V'(\tilde{R})] 
 \end{split}\end{equation}
 whose solution drives a simple harmonic oscillator
  \begin{equation}\label{FOMIMapprox0SHO}
  \begin{split}
 \dot{\tilde{r}} &= \tilde{p}-\tilde{P}\\
 \mu \dot{\tilde{p}} &= -\kappa \tilde{r}.
 \end{split}
 \end{equation}
This is very similar to the approximation from Section \ref{FPUT approx}.
%The external variables $(\tilde{R},\tilde{P})$ again solve
%the regular old FPUT system with no influence from the internal oscillators. 
The key difference is that 
 instead of $\tilde{r} = 0$ and $\tilde{p} = \tilde{P}$ as in Corollary \ref{naive},
the internal oscillators 
solve their equations of motion exactly with the caveat
that they are driven by what is now an approximate version of ${P}$. 
%In the Corollary \ref{naive}, $\tilde{r} = 0$ and $\tilde{p} = \tilde{P}$, so this just slightly different than what we had there. But now we can get a slightly better estimate. 

As described in Section \ref{general strategy} all the residuals apart from 
the second 
are zero, which is 
 $
 \Res^\mu_2(\tilde{\Phi}^\mu) = \kappa F^\mu[\tilde{r}(0),\tilde{p}(0),\tilde{P}]. 
  $
  Using \eqref{E est} and \eqref{F0} we have:
  \bes\begin{split}
  \| \Res^\mu_2(\tilde{\Phi}^\mu(t)) \| &
  \le C\left( \|\tilde{r}(0)\| + \sqrt{\mu } \| \tilde{p}(0) - \tilde{P}(0)\|\right)\\
  &+  C\mu \left(\|\dot{P}(t)\| + \|\dot{P}(0)\|
 + |t|\sup_{|t|\le T_*}\| \ddot{\tilde{P}}(t)\| \right).
 \end{split}\ees

Because it is part of the solution of FPUT, $\tilde{P}$  satisfies a global 
in time estimate like \eqref{FPUT control}. A routine bootstrap argument 
can be used to get global in time control of all higher order time derivatives of 
$\tilde{P}$ as well.
Therefore the final term above is genuinely $\O(\mu)$ for $|t| \le T_*$.
If we additionally demand that $\|\tilde{r}(0)\| +\sqrt{\mu} \| \tilde{p}(0) - \tilde{P}(0)\| \le C\mu$
 then we have $\|\Res^\mu_2(\tilde{\Phi}^\mu)\| \le C {\mu}$ on $[-T_*,T_*]$. 
Theorem \ref{gen approx} tell us the error of the approximation \eqref{FOMIMapprox0FPUT}-\eqref{FOMIMapprox0SHO} is $\O({\mu})$, a half power of $\mu$
 better than 
 in Corollary \ref{naive}. Here is the rigorous result:
 \begin{cor} \label{naive2}
Let $\kappa>0$, $K_*>0$, $T_*>0$ and $V: \R \to \R$ be smooth with $V(0) = V'(0) = 0$
and $V''(0) =: k > 0$.  Then there exists
$\rho_* = \rho_*(V)>0$, $\mu _* = \mu_*(K_*,T_*,\kappa,V)>0$ and $C_* = C_*(K_*,T_*,\kappa,V)>0$ 
for which
we have the following when $\mu \in (0,\mu_*]$.

Suppose that $(\tilde{R},\tilde{P})$ solves the FPUT system \eqref{FOMIMapprox0FPUT}
 with 
$$
\|\tilde{R}(0)\| + \|\tilde{P}(0)\| \le \rho_*
$$
and $(\tilde{r},\tilde{p})$ solve the driven simple harmonic oscillator
\eqref{FOMIMapprox0SHO} with
$$
\|\tilde{r}(0)\| +\sqrt{\mu} \| \tilde{p}(0) - \tilde{P}(0)\|\le K_* {\mu}.
$$ 
Furthermore suppose that
$(R,P,r,p)$ solves the MiM lattice \eqref{FOMIM} with 
$$
\|(R(0),P(0),r(0),p(0)) - (\tilde{R}(0),\tilde{P}(0),\tilde{r}(0),\tilde{P}(0))\|_\mu \le {\mu}.
$$
Then
$$
\|(R(t),P(t),r(t),p(t)) - (\tilde{R}(t),\tilde{P}(t),\tilde{r}(t),\tilde{P}(t))\|_\mu \le C_* {\mu}
$$
for all $t \in [-T_*,T_*]$.
\end{cor}

\subsection{The higher order FPUT approximation} 
Going to next order of the approximation has a surprising outcome: the approximation
remains an FPUT approximation. 
Applying \eqref{I expand} with $m = 1$ to 
 $F^\mu[r(0),p(0),P]$ gets us, after some algebra,
\be\begin{split}
F^\mu[r(0),p(0),P] = &-{1 \over \omega_\mu^2}\dot{P} \\
+&\left(r(0)+{1 \over \omega_\mu^2}  
\dot{P}(0) \right)\cos(\omega_\mu t)\\ +& {1 \over \omega_\mu} \left( p(0) - P(0)
+{1 \over \omega_\mu^2} \ddot{P}(0)
\right) \sin(\omega_\mu t) \\
 +& \mathcal{E}_{2}^\mu[\dot{P}](t).
\end{split}\ee

%The chain rules tells us that
%$
%{\partial_t}\left(\de^-[V'(R)]\right)=\de^-[V''(R) \dot{R}].
%$
%Since $\dot{R} = P$ this becomes
%$
%{\partial_t}\left(\de^-[V'(R)]\right)=\de^-[V''(R) P].
%$
%Thus we have the expansion:
%\be\begin{split}
%F^\mu[r(0),p(0),P] = &-{1 \over \omega_\mu^2} \de^-[V'(R)] \\
%+&\left(r(0)+{1 \over \omega_\mu^2}  
%\de^-[V'(R(0))] \right)\cos(\omega_\mu t)\\ +& {1 \over \omega_\mu} \left( p(0) - P(0)
%+{1 \over \omega_\mu^2} \de^-[V''(R(0)) P(0)] 
%\right) \sin(\omega_\mu t) \\
% +& \mathcal{E}_{2}^\mu[\de^-[V'(R)]](t).
%\end{split}\ee
We can make the second two lines as small as we please by imposing restrictions on the 
initial data and the last line is expected to be $\O(\mu^2)$. Thus we are lead to the choice of 
 $$\ds \tilde{F}^\mu = -{1 \over \omega_\mu^2} \dot{P}=-{\mu \over \kappa} \dot{P}.$$
 With, this (and some really easy algebra) we form an approximating system from \eqref{FOMIMapprox}. The variables $(\tilde{R},\tilde{P})$  solve 
%\bes
% \begin{split}
% \dot{\tilde{R}} &= \delta^+ \tilde{P}\\
% \dot{\tilde{P}} &=\delta^- [V'(\tilde{R})]-{\kappa \over \omega_\mu^2} \de^-[V'(\tilde{R})]
% \\
% \dot{\tilde{r}} &= \tilde{p}-\tilde{P}\\
% \mu \dot{\tilde{p}} &= -\kappa \tilde{r}.
% \end{split}
% \ees
% Recalling the definition of $\omega_\mu^2$ we see that
 %$1-\kappa/\om_\mu^2 = 1/(1+\mu)$ so the above now reads
\be\label{FOMIMapprox2FPUT}
 \begin{split}
 \dot{\tilde{R}} &= \delta^+ \tilde{P}\\
 \dot{\tilde{P}} &={1 \over 1+\mu} \delta^- [V'(\tilde{R})]
\end{split} \ee
and the variables $(\tilde{r},\tilde{p})$ solve
 \be\label{FOMIMapprox2SHO}\begin{split}
 \dot{\tilde{r}} &= \tilde{p}-\tilde{P}\\
 \mu \dot{\tilde{p}} &= -\kappa \tilde{r}.
 \end{split}
 \ee
These are, again {\it barely different} that the FPUT approximations \eqref{FOMIMapprox} or \eqref{FOMIMapprox0FPUT}-\eqref{FOMIMapprox0SHO}.
The $(\tilde{R},\tilde{P})$ system
 \eqref{FOMIMapprox2FPUT} is once more FPUT,  but the potential function is slightly modified by the factor $1/(1+\mu)$, a roughly $\O(\mu)$ change. 
 %The surprising thing is that we can eke out an extra $\mu^{3/2}$ in the error estimate
%now, over what we had in FPUT approximations described in Corollaries \ref{naive} and \ref{naive2}. The cost is more restrictions on the initial conditions.

To wit, we compute the residuals. As we saw above in Section \ref{general strategy}, only $\Res_2(\tilde{\Phi}^\mu)$ is non-zero and in this setting is given by
\be\begin{split}
\Res_2(\tilde{\Phi}^\mu)= &\kappa\left(\tilde{r}(0)+{1 \over \omega_\mu^2}  
\dot{\tilde{P}}(0) \right)\cos(\omega_\mu t)\\ +& {\sqrt{\mu}} \left( \tilde{p}(0) - \tilde{P}(0)
+{1 \over \omega_\mu^2} \ddot{\tilde{P}}(0)
\right) \sin(\omega_\mu t) \\
 +& \kappa \mathcal{E}_{2}^\mu[\dot{\tilde{P}}](t).
\end{split}\ee
Since $(\tilde{R},\tilde{P})$ satisfy an FPUT system, we get global in time estimates
for them as in \eqref{FPUT control}; that there is a mild $\mu$ dependence in the equations
for $(\tilde{R},\tilde{P})$ does not effect this estimate in any way, so long as $\mu$ is not too big. And, as in the previous section, it is elementary to bootstrap and get $\mu$-uniform estimates on $\dot{\tilde{P}}$, $\ddot{\tilde{P}}$ and so on. Thus if we apply \eqref{E est} we find \bes
\| \mathcal{E}_2^\mu[\dot{\tilde{P}}](t)\| \le C\omega_\mu^{-n-2}\left( \| P^{(4)}(t)\| +\| P^{(4)}(0)\|+|T_*|\sup_{t'\le |t|}\| P^{(5)}(t')\| \right) \le C \mu^2.
\ees

Then we demand  
$$
\left\|\tilde{r}(0)+{1 \over \omega_\mu^2}  
\dot{\tilde{P}}(0)\right\|+  {\sqrt{\mu}}\left\| \tilde{p}(0) - \tilde{P}(0)
+{1 \over \omega_\mu^2} \ddot{\tilde{P}}(0)\right\| \le C \mu^2.
$$
In which case we now have $\|\Res_2(\tilde{\Phi}^\mu)\| \le C\mu^2.$
Since $\dot{\tilde{P}} = (1+\mu)^{-1} \de^-[V'(\tilde{R})]$ we can rewrite the above condition
in a slightly more functional way as
$$
\left\|\tilde{r}(0)+{\mu \over \kappa(1+\mu)}  
 \de^-[V'(\tilde{R}(0))]\right\|+  {\sqrt{\mu}}\left\| \tilde{p}(0) - \tilde{P}(0)
+{\mu \over \kappa(1+\mu)}   \de^-[V''(\tilde{R}(0)) \de^+\tilde{P}(0)]\right\| \le C \mu^2.
$$
And the geometric series tells us that the above is implied by
$$
\left\|\tilde{r}(0)+{\mu \over \kappa}  
 \de^-[V'(\tilde{R}(0))]\right\|+  {\sqrt{\mu}}\left\| \tilde{p}(0) - \tilde{P}(0)
+{\mu \over \kappa}   \de^-[V''(\tilde{R}(0)) \de^+\tilde{P}(0)]\right\| \le K_* \mu^2.
$$
With all of the above considerations, we can invoke Theorem \ref{gen approx}:
\begin{cor}\label{totally awesome}
Let $\kappa>0$, $K_*>0$, $T_*>0$ and $V: \R \to \R$ be smooth with $V(0) = V'(0) = 0$
and $V''(0) =: k > 0$.  Then there exists
$\rho_* = \rho_*(V)>0$, $\mu _* = \mu_*(K_*,T_*,\kappa,V)>0$ and $C_* = C_*(K_*,T_*,\kappa,V)>0$ 
for which
we have the following when $\mu \in (0,\mu_*]$.

Suppose that $(\tilde{R},\tilde{P})$ solves the FPUT system
 \eqref{FOMIMapprox2FPUT} with
$$
\|\tilde{R}(0)\| + \|\tilde{P}(0)\| \le \rho_*
$$
and $(\tilde{r},\tilde{p})$ solve the driven simple harmonic oscillator
 \eqref{FOMIMapprox2SHO}
subject to
$$
\left\|\tilde{r}(0)+{\mu \over \kappa}  
 \de^-[V'(\tilde{R}(0))]\right\|+  {\sqrt{\mu}}\left\| \tilde{p}(0) - \tilde{P}(0)
+{\mu \over \kappa}   \de^-[V''(\tilde{R}(0)) \de^+\tilde{P}(0)]\right\| \le K_* \mu^2.
$$
Furthermore suppose that 
$(R,P,r,p)$ solves the MiM lattice \eqref{FOMIM} with 
$$
\|(R(0),P(0),r(0),p(0)) - (\tilde{R}(0),\tilde{P}(0),\tilde{r}(0),\tilde{P}(0))\|_\mu \le K_* \mu^2.
$$
Then
$$
\|(R(t),P(t),r(t),p(t)) - (\tilde{R}(t),\tilde{P}(t),\tilde{r}(t),\tilde{P}(t))\|_\mu \le C_* \mu^2
$$
for all $t \in [-T_*,T_*]$.

\end{cor}

\subsection{Challenges at the next order} Does this strategy always yield an FPUT
system whose solutions drive the internal oscillators? Put $m=2$ into \eqref{I expand}.
\be\begin{split}
F^\mu[r(0),p(0),P] = &-{1 \over \omega_\mu^2}\dot{P} + {1 \over \omega_\mu^4} \partial_t^3{P}\\
+&\left(r(0)+{1 \over \omega_\mu^2}  
\dot{P}(0)- {1 \over \omega_\mu^4} \partial_t^3{P}(0) \right)\cos(\omega_\mu t)\\ +& {1 \over \omega_\mu} \left( p(0) - P(0)
+{1 \over \omega_\mu^2} \ddot{P}(0)
-{1 \over \omega_\mu^4} \partial_t^4{P}(0)
\right) \sin(\omega_\mu t) \\
 +& \mathcal{E}_{2}^\mu[\dot{P}](t).
\end{split}\ee
If we followed the earlier strategy, we would truncate after the first line and use initial
data restriction and \eqref{E est} to control errors from  the last two. Imagine that we do this now, then our approximating system reads:
\be
 \begin{split}
 \dot{\tilde{R}} &= \delta^+ \tilde{P}\\
-{\mu^2 \over \kappa} \partial_t^3 \tilde{P}+(1+\mu)\dot{\tilde{P}} &=\delta^- [V'(\tilde{R})]
 \\
 \dot{\tilde{r}} &= \tilde{p}-\tilde{P}\\
 \mu \dot{\tilde{p}} &= -\kappa \tilde{r}.
 \end{split}
 \ee
 Again the first two lines are self-contained, but are not an FPUT system---they are a singularly perturbed FPUT equation. It is not at all obvious that such an approximation is useful, since the approximating system is now as complex as the original. We go no further.

\bibliographystyle{siam}
\bibliography{mim-approx-bib}{}
\end{document}